\documentclass[a4paper]{amsart}
\usepackage{amssymb}
\usepackage{ifthen}
\usepackage{graphicx}

\newtheorem{thm}{Theorem}

\newtheorem{lem}{Lemma}

\newtheorem{rem}{Remark}

\theoremstyle{definition}

\newcommand{\A}{{\mathcal A}}

\newcommand{\es}{{\mathcal S}}
\newcommand{\K}{{\mathcal K}}
\newcommand{\D}{{\mathbb D}}

\begin{document}
\bibliographystyle{amsplain}

\title[Simple proofs of certain inequalities with logarithmic coefficients]{Simple proofs of certain inequalities with logarithmic coefficients of univalent functions}

\author[M. Obradovi\'{c}]{Milutin Obradovi\'{c}}
\address{Department of Mathematics,
Faculty of Civil Engineering, University of Belgrade,
Bulevar Kralja Aleksandra 73, 11000, Belgrade, Serbia.}
\email{obrad@grf.bg.ac.rs}

\author[N. Tuneski]{Nikola Tuneski}
\address{Department of Mathematics and Informatics, Faculty of Mechanical Engineering, Ss. Cyril and
Methodius
University in Skopje, Karpo\v{s} II b.b., 1000 Skopje, Republic of North Macedonia.}
\email{nikola.tuneski@mf.edu.mk}


\begin{abstract}
In this paper, we give simple proofs for the bounds (some of them sharp) of the difference of the moduli of the second and the first logarithmic coefficient for the general class of univalent functions and for the class of convex univalent functions.
\end{abstract}

\maketitle

\section{Introduction and definitions}

Let $\A$ be the class of functions $f$ analytic in the open unit disk $\D=\{z:|z|<1\}$ and normalized such that $f(z)=z+a_2z^2+a_3z^3+\cdots$. If additionally, $f$ is one-on-one and onto, we say that it is univalent and denote with $\es$ the corresponding subclass of $\A$ containing all such functions.

One of the most important results of the twentieth century, the Bieberbach conjecture is related with this class, and says that for all univalent functions, $|a_n|\le n$ for all positive integers $n$. It was formulated in 1916 (\cite{bieber}) and proven in 1985 by de Branges \cite{branges}. Huge part of the theory of univalent functions is oriented towards finding estimates (preferably sharp) of expressions involving moduli of coefficients $a_n$. More details can be found in \cite{book}. In parallel, significant attention is given to expressions involving so-called logarithmic coefficients, $\gamma_n$, $n=1,2,\ldots$, defined  by
\begin{equation}\label{eq2}
\log\frac{f(z)}{z}=2\sum_{n=1}^\infty \gamma_n z^n.
\end{equation}
The expressions are studied over the general class $\es$ of univalent functions, or over its subclasses (starlike, convex, close-to-convex, et c.).

In this paper we give estimates of the difference of the moduli of the second and the first logarithmic coefficient,
\[|\gamma_2|-|\gamma_1|,\]
for the general class of univalent functions and for the class of convex functions. In the first case the estimate is sharp, while in the second is partly sharp. The class of convex functions, $\K$,  consists of functions $f$ from $\A$ that map the open unit disk $\D$ onto a convex domain. These functions are univalent.

The result for the general class of univalent functions is sharp and was previously obtained in \cite[Theorem 3.1]{lecko} and is rather complicated. Here we give a simple proof. The result for the class of convex functions is partly sharp.

Before continuing, let note that by equating the coefficients on both sides of \eqref{eq2} we receive
\begin{equation}\label{eqc}
\gamma_1=\frac{a_2}{2} \quad \text{and} \quad \gamma_2 = \frac12\left( a_3-\frac{a_2^2}{2} \right).
\end{equation}

\section{The result for the general class $\es$}

In \cite{lecko} the authors used robust technique to obtain sharp estimate of $|\gamma_2|-|\gamma_1|$ for the general class of univalent functions. Here we will prove the same using elementary technique and the well known sharp inequality holding for all $f$ from $\es$:
\begin{equation}\label{e3}
|a_3-a_2^2|\le1.
\end{equation}
For its proof see \cite[p.5]{book}.

\begin{thm}
For every function $f\in\es$, $-\frac{\sqrt2}{2} \le |\gamma_2|-|\gamma_1| \le \frac12$ holds sharply.
\end{thm}

\begin{proof}
From the Bieberbach conjecture, for all $f\in\es$, we have $\frac12|a_2|\le1$, and further, applying \eqref{e3} in the final step, we have
\[
\begin{split}
|\gamma_2|-|\gamma_1| &= \frac12\left| a_3-\frac{a_2^2}{2} \right| -\frac12|a_2|  \\
&\le \frac12\left| a_3-\frac{a_2^2}{2} \right| -\frac12|a_2| \left(\frac12|a_2| \right)\\
&\le \frac12\left( \left|a_3-\frac{a_2^2}{2}\right| -\frac12|a_2|^2\right) \\
&\le \frac12\left| \left(a_3-\frac{a_2^2}{2}\right) -\frac12 a_2^2\right| \\
&= \frac12\left| a_3- a_2^2 \right| \le \frac12.
\end{split}
\]
The estimate is sharp since equality is attained for the function
\[f_1(z) = \frac{z}{1+e^{i\theta} z^2} = z-e^{i\theta} z^3 +e^{2i\theta}z^5+\cdots.\]

The second inequality that is to be proven, $ |\gamma_2|-|\gamma_1| \ge -\frac{\sqrt2}{2} $,  is equivalent to $\frac12\left| a_3-\frac{a_2^2}{2}\right| - \frac12 |a_2| \ge -\frac{\sqrt2}{2}$, i.e., to
\begin{equation}\label{e2}
\left| a_3-\frac{a_2^2}{2}\right| \ge |a_2| - \sqrt2.
\end{equation}
If $|a_2| < \sqrt2$, then \eqref{e2} is obviously true. For the remaining case, when $\sqrt2 \le |a_2|\le2$, we have
\[
\begin{split}
\left| a_3-\frac{a_2^2}{2}\right| &= \left| \left(a_3-a_2^2\right) + \frac{a_2^2}{2}\right| \ge \frac{|a_2|^2}{2} - \left|a_3-a_2^2\right| \ge \frac{|a_2|^2}{2} - 1 \ge |a_2|-\sqrt2.
\end{split}
\]
The last inequality holds since it is equivalent to $(|a_2|-\sqrt2)(|a_2|+\sqrt2-2)\ge0$, and also $|a_2|-\sqrt2 \ge0$ and $|a_2|+\sqrt2-2 >0$. The estimate is sharp with the equality attained for the function
\[f(z)=\frac{z}{1-\sqrt2e^{i\theta}z+e^{2i\theta}z^2} = z + \sqrt2e^{i\theta}z^2 + e^{2i\theta}z^3+\cdots\]
 for which $|a_2|=\sqrt2$ and $|a_3-a_2^2|=1$.
\end{proof}

\section{The result for the class $\K$}

Now we will prove a partly sharp estimate of $|\gamma_2|-|\gamma_1| $ over the class $\K$, and for that we will make use of the following result due to \cite{trimble}.

\begin{lem}\label{lem-2}
For all functions $f(z)=z+a_2z^2+a_3z^3+\cdots$ from $\K$,
\[\left|a_3-a_2^2\right|\le \frac13\left( 1-|a_2|^2 \right).\]
The inequality is sharp with extremal function
\[ f_{\lambda}(z) = \int_0^z\left( \frac{1-t}{1+t} \right)^{\lambda}\frac{1}{1-t^2}\,dt = z+\lambda z^2+\frac13\left( 2\lambda^2+1\right)z^3+\cdots, \]
with $0\le\lambda\le1$.
\end{lem}

Note that the range of $\lambda$ for the extremal function exploits the fact that all coefficients in the expansion of  convex univalent functions have modulus less or equal to 1.

\begin{thm}
For every function $f\in\K$, $-\frac{1}{\sqrt{10}} \le |\gamma_2|-|\gamma_1| \le \frac16$. The second inequality is sharp.
\end{thm}

\begin{proof}
From Lemma \ref{lem-2} we have
\[
\begin{split}
|\gamma_2|-|\gamma_1| &= \frac12\left| a_3-\frac{a_2^2}{2} \right| -\frac12|a_2|  \\
&= \frac12\left| \left(a_3-a_2^2\right) + \frac{a_2^2}{2}\right| -\frac12|a_2| \\
&\le \frac12\left| a_3-a_2^2\right| + \frac14|a_2|^2 -\frac12|a_2| \\
&\le \frac16\left( 1-|a_2|^2 \right) +\frac14|a_2|^2-\frac12|a_2|\\
&\le \frac{1}{12}\left( |a_2|^2 -6|a_2|+2\right)\\
&\le \frac{1}{6},
\end{split}
\]
since $|a_2|\le1$ for $f\in \K$. Equality holds for the function $f_0(z)=z+\frac13z^3+\cdots$, where $f_0$ is given in Lemma \ref{lem-2} for $\lambda=0$.

For the remaining case we need to prove
\[ |\gamma_2|-|\gamma_1| = \frac12\left| a_3-\frac{a_2^2}{2} \right| -\frac12|a_2| \ge -\frac{1}{\sqrt{10}}, \]
i.e.,
\begin{equation}\label{eq-l}
\left| a_3-\frac{a_2^2}{2} \right| \ge |a_2| -\frac{2}{\sqrt{10}}.
\end{equation}
If $|a_2| < \frac{2}{\sqrt{10}}$, the above obviously holds, while for $\frac{2}{\sqrt{10}} \le |a_2|\le1$, we have
\[
\begin{split}
\left| a_3-\frac{a_2^2}{2} \right| &= \left| \left(a_3-a_2^2\right) + \frac{a_2^2}{2}\right| \\
&\ge \frac{|a_2|^2}{2} - \left|a_3-a_2^2\right| \\
&\ge \frac{|a_2|^2}{2} - \frac13\left(1-|a_2|^2\right) \\
&= \frac56 |a_2|^2-\frac13 \\
&\ge |a_2|-\frac{2}{\sqrt{10}}.
\end{split}
\]
The last inequality holds since it is equivalent to $5|a_2|^2-6|a_2|+\frac{12}{\sqrt{10}}-2\ge0$ which is easily verified to be true.
\end{proof}

\begin{rem}
We were not able to prove sharpness of the left estimate in the previous theorem and it remains an open problem.
\end{rem}

\end{document}